\documentclass{amsart}
\usepackage{amsmath,amssymb,cite,geometry,enumerate}
\usepackage[colorlinks,citecolor=blue,linkcolor=blue,urlcolor=blue]{hyperref}
\newtheorem{theorem}{Theorem}[section]

\newtheorem{proposition}[theorem]{Proposition}
\newtheorem{corollary}[theorem]{Corollary}

\newtheorem{remark}[theorem]{Remark}
\numberwithin{equation}{section}
\allowdisplaybreaks[4]

\begin{document}

\title[Asymptotic Expansion of the Heat Trace of the Thermoelastic DN map]{Asymptotic Expansion of the Heat Trace of the Thermoelastic Dirichlet-to-Neumann map}

\author{Genqian Liu}
\address{School of Mathematics and Statistics, Beijing Institute of Technology, Beijing 100081, China}
\email{liugqz@bit.edu.cn}

\author{Xiaoming Tan}
\address{School of Mathematics and Statistics, Beijing Institute of Technology, Beijing 100081, China}
\email{xtan@bit.edu.cn}


\begin{abstract}
    This paper is devoted to study the asymptotic expansion of the heat trace of the Dirichlet-to-Neumann map for the thermoelastic equation on a Riemannian manifold with doundary. By providing a method we can obtain all the coefficients of the asymptotic expansion. In particular, we explicitly give the first two coefficients involving the volume and the total mean curvature of the boundary.
\end{abstract}

\maketitle

\section{Introduction}

\addvspace{5mm}

Let $\Omega$ be a smooth compact Riemannian manifold of dimension $n$ with smooth boundary $\partial \Omega$. We consider $\Omega$ as a homogeneous, isotropic, thermoelastic body. Assume that the Lam\'{e} coefficients $\lambda$, $\mu$ and the heat conduction coefficient $\alpha$ of the thermoelastic body are constants which satisfy $\mu > 0$, $\lambda + \mu \geqslant 0$ and $\alpha>0$. We denote by $\operatorname{grad},\, \operatorname{div},\, \Delta_{g},\, \Delta_{B}$ and $\operatorname{Ric}$, respectively, the gradient operator, the divergence operator, the Laplace--Beltrami operator, the Bochner Laplacian and the Ricci tensor. We define the thermoelastic operator
\begin{align}\label{1.2}
    \mathcal{L}_g :=
    \begin{pmatrix}
        \mu \Delta_{B} + (\lambda + \mu)\operatorname{grad}\operatorname{div} + \mu \operatorname{Ric} + \rho\omega^2 & -\beta\operatorname{grad} \\
        i\omega\theta_{0}\beta\operatorname{div} & \alpha\Delta_{g} + i\omega \gamma
    \end{pmatrix},
\end{align}
where the coefficient $\beta$ depends on $\lambda,\,\mu$ and the linear expansion coefficient, $\gamma$ is the specific heat per unit volume, $\theta_{0}$ is the reference temperature, $\rho$ is the density of the thermoelastic body, $\omega$ is the angular frequency and $i = \sqrt{-1}$. We denote by $\textbf{\textit{u}}(x) \in (C^{\infty}(\Omega))^{n}$ and $\theta(x) \in C^{\infty}(\Omega)$ the displacement and the temperature variation, respectively. By the methods of \cite{Liu21,Liu19} (cf. \cite{Bao19,Biot56,Carlson72,LandLif86}, \cite[Ch.\,X]{Kupr80} and \cite[Ch.\,6]{LiQin14}), we can write the boundary value problem for the steady thermoelastic equation as follows:
\begin{align}\label{1.1}
    \begin{cases}
        \mathcal{L}_g \textbf{\textit{U}} = 0 & \text{in}\ \Omega, \\
        \textbf{\textit{U}} = \textbf{\textit{V}} \quad & \text{on}\ \partial \Omega,
    \end{cases}
\end{align}
where $\textbf{\textit{U}} = (\textbf{\textit{u}},\theta)^t$ and the superscript $t$ denotes the transpose. Problem \eqref{1.1} is an extension of the boundary value problem for classical elastic equation. In particular, when $\Omega$ is a bounded domain in $\mathbb{R}^n$ and the temperature of the medium is not taken into consideration, problem \eqref{1.1} reduces to the corresponding problem for classical elastic equation.

\addvspace{2mm}

The theory of thermoelasticity describes the behaviors of elastic bodies under the influence of nonuniform temperature fields. It is a generalization of the theory of elasticity. In the special case of classical elasticity, assume that the temperature of the medium is the same at all points and is not changed during deformation. Nevertheless, in reality, most materials are not pure elastic, a temperature change in the elastic body will generate additional strain and stress. Therefore, deformation is followed by temperature variation and, conversely, temperature variation is followed by deformation of the thermoelastic body due to thermal expansion. Even if there is no external stresses or mass forces, some heat sources change the temperature of the thermoelastic body, there will be deformation (see \cite[p.\,36]{Kupr80}). The thermoelastic behavior of materials is revealed by everyday experience in the form of linear or bulk expansions and contractions under the effect of temperature changes. The material responds to mechanical or thermal excitation in an instantaneous manner and, when the excitation is removed, returns to its initial state without showing any memory of the recent changes (see \cite[p.\,301]{Salenco01}). In thermoelasticity, stress, strain, and temperature are interrelated in a very complicated manner. The exact solution of problems of general thermoelasticity presents, therefore, enormous difficulties, and the number of known solutions is very small (see \cite[p.\,1]{Parkus76}). The significance of the problem \eqref{1.1} consists not only in the important independent role of this state, frequently occurring in engineering, geophysics, plasma physics, and related topics, but also in the fact that their investigation opens up the way to the study the general type (see \cite[p.\,528]{Kupr80}). 

\addvspace{2mm}

In the present paper, we consider the following thermoelastic Steklov problem:
\begin{equation}\label{1.5}
    \begin{cases}
        \mathcal{L}_g \textbf{\textit{U}} = 0 \quad & \text{in}\ \Omega, \\
        \Lambda_{g} (\textbf{\textit{U}}) = \tau \textbf{\textit{U}} & \text{on}\ \partial \Omega,
    \end{cases}
\end{equation}
where the thermoelastic Dirichlet-to-Neumann map $\Lambda_{g}$ associated with the operator $\mathcal{L}_g$ is defined by
\begin{align}\label{1.4}
    \Lambda_{g} (\textbf{\textit{V}}) :=
    \begin{pmatrix}
        \lambda\nu \operatorname{div} + \mu\nu S & -\beta \nu \\
        0 & \alpha\partial_\nu
    \end{pmatrix}
    \textbf{\textit{U}} \quad \text{on}\ \partial \Omega,
\end{align}
where $\textbf{\textit{U}}$ solves the problem \eqref{1.1}, $\nu$ is the outward unit normal vector to $\partial \Omega$ and $S$ is the stress tensor (also called the deformation tensor) of type $(1,1)$ defined by
\begin{align}\label{1.7}
    S\textbf{\textit{u}}:=\nabla \textbf{\textit{u}} + \nabla \textbf{\textit{u}}^t, \quad \textbf{\textit{u}} \in (C^1(\Omega))^n,
\end{align}
see \eqref{2.12} (or \cite[p.\,562]{Taylor11.3}) for details. According to the theory of elliptic operators, $\Lambda_{g}$ is a self-adjoint, first order, elliptic pseudodifferential operator (see Section \ref{s3}), the spectrum of $\Lambda_{g}$ consists of a discrete sequence
\begin{align*}
    0 \leqslant \tau_{1} \leqslant \tau_{2} \leqslant \cdots \leqslant \tau_{k} \leqslant \cdots \to +\infty
\end{align*}
with each eigenvalue repeated according to its multiplicity. Let $\{\textbf{\textit{U}}_k\}_{k \geqslant 1}$ be the eigenvectors corresponding to the eigenvalues $\{\tau_k\}_{k \geqslant 1}$. Then the eigenvectors $\{\textbf{\textit{U}}_k\}_{k \geqslant 1}$ form an orthogonal basis in $(L^{2}(\partial \Omega))^{n+1}$. Notice that the thermoelastic Steklov spectrum coincides with the spectrum of the thermoelastic Dirichlet-to-Neumann map, they are physical quantities which can be measured experimentally.

There is an important problem states that what geometric information about manifolds can be explicitly obtained by providing the thermoelastic Steklov spectrum? This is analogous to the well-known Kac's problem \cite{Kac66} (i.e., is it possible to ``hear'' the shape of a domain just by ``hearing'' all the eigenvalues of the Dirichlet Laplacian? See also \cite{Loren47,Protter87,Weyl12}). Once the corresponding thermoelastic problem is solved, one can obtain the geometric quantities of a manifold from the thermoelastic Steklov spectrum. In this paper, we give an affirmative answer for this problem.

The thermoelastic Steklov problem \eqref{1.5} originates from inverse spectral problems. One hopes to recover the geometry of a manifold from the set of the known data (the spectrum of a differential or pseudodifferential operator). Here we briefly recall the classical Steklov problem associated with the Laplace--Beltrami operator as follows:
\begin{equation*}
    \begin{cases}
        \Delta_g u = 0  & \text{in}\ \Omega, \\
        \frac{\partial u}{\partial \nu} = \eta u \quad& \text{on}\ \partial \Omega.
    \end{cases}
\end{equation*}
The classical Dirichlet-to-Neumann map $\mathcal{D}$ (also known as the voltage-to-current map) is defined by $\mathcal{D}(f) := \frac{\partial u}{\partial \nu}|_{\partial \Omega}$, where $u$ solves the Dirichlet problem: $\Delta_g u = 0$ in $\Omega$, $u = f$ on $\partial \Omega$. The study of the
spectrum of $\mathcal{D}$ was initiated by Steklov in 1902 (see \cite{Stekloff02}). It follows from a famous result of \cite{Sandgren55} that the eigenvalue counting function $N(\eta):=\#(k|\eta_{k} \leqslant \eta)$ satisfies the Sandgren's asymptotic formula
\begin{equation}\label{1.6}
    N(\eta) := \frac{\operatorname{vol}(B^{n-1})\operatorname{vol}(\partial \Omega)}{(2\pi)^{n-1}} \eta^{n-1} + o(\eta^{n-1}) \quad \text{as}\ \eta \to +\infty,
\end{equation}
or equivalently,
\begin{equation*}
    \sum_{k=1}^{\infty} e^{-t \eta_{k}} 
    = \operatorname{Tr} (e^{-t \mathcal{D}})
    \sim \frac{\Gamma(n) \operatorname{vol}(B^{n-1})}{(2\pi)^{n-1}t^{n-1}} \operatorname{vol}(\partial \Omega) \quad \text{as}\ t \to 0^+, 
\end{equation*}
where $B^{n-1}$ is the $(n-1)$-dimensional unit ball in $\mathbb{R}^{n-1}$. This shows that one can obtain the volume $\operatorname{vol}(\partial \Omega)$ of the boundary from the first term of the above asymptotic expansion. The first author of this paper gave a sharp remainder $O(\eta^{n-2})$ of \eqref{1.6} in \cite{Liu11}.

\addvspace{2mm}

Let us point out that the thermoelastic operator $\mathcal{L}_g$ given by \eqref{1.2} is a more complicated differential operator than the elastic operator, its second-order terms are not merely the Laplacian or even the Laplace--Beltrami operator. Indeed, the thermoelastic operator $\mathcal{L}_g$ is a non-Laplace type operator, such operator has a wide range of applications (see, for example, \cite{Full92,KawohlLeVe93,Kupr65,LandLif86,Sommerfeld64}). Contrary to the Laplace type operators, there are no systematic effective methods to explicitly calculate the coefficients of the asymptotic expansions of the traces for non-Laplace type operators. In 1971, Greiner \cite[p.\,164]{Greiner71} indicated that {\it “the problem of interpreting these coefficients geometrically remains open”}. In this sense, it is just beginning that the study of geometric aspects of spectral asymptotics for non-Laplace type operators, the corresponding methodology is still underdeveloped in comparison with the theory of the Laplace type operator. Thus, {\it the geometric aspect of the spectral asymptotics of non-Laplace type operators remains an open problem} (see \cite{Avra06}).

Due to the fact that the Riemannian structure on a manifold is determined by a Laplace type operator, the systematic explicit calculations of coefficients of heat kernels for Laplace type operators are now well understood, see Gilkey \cite{Gilkey75} and many others (see \cite{AGMT10,BGV92,Gilkey79,Gilkey95,Grubb86,Grubb09,Kirsten01,Liu11,Vassilevich03,SaVa97} and references therein). For the classical boundary conditions, like Dirichlet, Neumann, Robin, and mixed combination thereof on vector bundles, the coefficients of the asymptotic expansions have been explicitly computed up to the first five terms (see, for example, \cite{BranGilk90,BPKV99,Kirs98}). For other types of operators, the first author of this paper gave the first two coefficients of the asymptotic expansions of the heat traces of the Navier--Lam\'{e} operator \cite{Liu21} and the Stokes operator \cite{Liu22S}. The authors of this paper gave the first two coefficients of the asymptotic expansions of the traces for the thermoelastic operators with the Dirichlet and Neumann conditions \cite{LiuTan22.1}. The first author of this paper also gave the first four coefficients of the asymptotic expansions of heat traces for the classical Dirichlet-to-Neumann map \cite{Liu15} (see Polterovich and Sher \cite{PoltSher15} for the first three coefficients) and the polyharmonic Steklov operator \cite{Liu22p}. Recently, the authors of this paper have obtained the first four coefficients of the asymptotic expansion of the heat trace for the magnetic Dirichlet-to-Neumann map associated with the magnetic Schr\"{o}dinger operator \cite{LiuTan21}. We refer the reader to \cite{Liu19} for the asymptotic distribution of Steklov eigenvalues of the elastic Dirichlet-to-Neumann map.

To obtain more geometric information about the manifold, we will study the thermoelastic Steklov problem \eqref{1.5} and the asymptotic expansion of $\operatorname{Tr} (e^{-t \Lambda_g})$ (the so-called ``heat kernel method''). The coefficients of the asymptotic expansion are spectral invariants which are metric invariants of the boundary of the manifold, they contain a lot of geometric and topological information about the manifold and asymptotic properties of the spectrum (see \cite{Edward91,Gilkey75,Kac66,Liu19}). Spectral invariants are of great importance in spectral geometry, they are also connected with many physical concepts (see \cite{ANPS09,Full95}) and have increasingly extensive applications in physics since they describe corresponding physical phenomena. However, computations of spectral invariants are challenging problems (see \cite{Gilkey75,Grubb86,Liu11,SaVa97}). The complexity of the thermoelastic operator directly lead to the calculations of some quantities (for example, spectral invariants) become much more difficult than that of the Laplace type operator. By the theory of pseudodifferential operators and symbol calculus, we can calculate all the coefficients $a_k\ (0 \leqslant k \leqslant n-1)$ of the asymptotic expansion, in particular, we explicitly give the first two coefficients $a_0$ and $a_1$. To the best of our knowledge, it has not previously been systematically applied in the context of the thermoelastic Dirichlet-to-Neumann map associated with the thermoelastic operator.

\addvspace{2mm}

The main result of this paper is the following theorem.
\begin{theorem}\label{thm1.1}
    Let $\Omega$ be a smooth compact Riemannian manifold of dimension $n$ with smooth boundary $\partial \Omega$, and $H$ be the mean curvature of the boundary $\partial \Omega$. Assume that the Lam\'{e} coefficients $\lambda$, $\mu$ and the heat conduction coefficient $\alpha$ are constants which satisfy $\mu > 0$, $\lambda + \mu > 0$ and $\alpha>0$. Let $\{\tau_k\}_{k \geqslant 1}$ be the eigenvalues of the thermoelastic Dirichlet-to-Neumann map $\Lambda_{g}$ associated with the thermoelastic operator $\mathcal{L}_g$. Then
    \begin{equation*}
        \sum_{k=1}^{\infty} e^{-t \tau_{k}} 
        = \operatorname{Tr} (e^{-t \Lambda_{g}})
        = \sum_{k=0}^{n-1} a_{k} t^{-n+k+1} + o(1)\quad \text{as}\ t \to 0^+, 
    \end{equation*}
    where the coefficients $a_k$ are spectral invariants which can be explicitly calculated for $0 \leqslant k \leqslant n-1$. In particular,
    \begin{align*}
        \sum_{k=1}^{\infty} e^{-t \tau_{k}} = t^{1-n} \int_{\partial \Omega} a_0(x) \,dS + t^{2-n} \int_{\partial \Omega} a_1(x) \,dS + 
        \begin{cases}
            O(t \log t), & n=2, \\
            O(t^{3-n}), & n \geqslant 3,
        \end{cases}
        \quad \text{as}\ t \to 0^+.
    \end{align*}
    Here,
    \begin{align*}
        a_0(x) = \frac{\Gamma(n-1)\operatorname{vol}(\mathbb{S}^{n-2})}{(2\pi)^{n-1}} 
        \biggl[ 
            \biggl( \frac{\lambda + 3\mu}{2\mu (\lambda + \mu)} \biggr)^{n-1} + \frac{1}{(2\mu)^{n-1}} + \frac{n-2}{\mu^{n-1}} + \frac{1}{\alpha^{n-1}}
        \biggr],
    \end{align*}
    and
    \begin{align*}
        a_1(x) & = \frac{\Gamma(n-1)\operatorname{vol}(\mathbb{S}^{n-2})H}{(2\pi)^{n-1}} 
        \biggl[ 
            \frac{n(n-2)}{2(n-1)\mu^{n-2}} + \frac{1}{2^{n-1}\mu^{n-2}} 
            \biggl( 1+\frac{\lambda^3+23\lambda^2\mu+87\lambda\mu^2+97\mu^3}{4(n-1)(\lambda+3\mu)^3}  \biggr)\\
            &\quad - \biggl( \frac{\lambda+3\mu}{2\mu(\lambda+\mu)} \biggr)^{n-1} 
            \biggl( \frac{2\lambda^3+9\lambda^2\mu+10\lambda\mu^2-\mu^3}{(\lambda+3\mu)^2} + \frac{\mu(15\lambda^3+105\lambda^2\mu+233\lambda\mu^2+175\mu^3)}{4(n-1)(\lambda+3\mu)^3} \biggl) \\
            &\quad +\frac{n-2}{2(n-1)\alpha^{n-2}} 
        \biggr],
    \end{align*}
    where $\operatorname{vol}(\mathbb{S}^{n-2})$ is the $(n-2)$-dimensional volume of the unit sphere $\mathbb{S}^{n-2} \subset \mathbb{R}^{n-1}$.
\end{theorem}

\addvspace{2mm}

Theorem \ref{thm1.1} is a generalization of the corresponding result for the elastic Dirichlet-to-Neumann map (see \cite{Liu19}). It is clear that the knowledge provided by the thermoelastic Dirichlet-to-Neumann map $\Lambda_g$  (i.e., the data set $\{\textbf{\textit{U}}|_{\partial \Omega},\Lambda_g(\textbf{\textit{U}}|_{\partial \Omega})\})$ is equivalent to the information given by all the thermoelastic Steklov eigenvalues and eigenvectors. Therefore, by Theorem \ref{thm1.1} we immediately have 

\begin{corollary}
    Both the volume $\operatorname{vol}(\partial \Omega)$  and the total mean curvature $\int_{\partial \Omega}H\,dS$ of the boundary $\partial \Omega$ are the thermoelastic Steklov spectral invariants.
\end{corollary}

By applying the Tauberian theorem (see {\rm\cite[p.\,30, Theorem 15.3]{Kore04}} or {\rm\cite[p.\,107]{Taylor11.2}}), we can easily obtain the following Weyl-type law for the counting function of the thermoelastic Steklov eigenvalues:

\begin{corollary}
    Let $N(\tau):= \#\{k|\tau_k \leqslant \tau\}$ be the counting function of the thermoelastic Steklov eigenvalues. Then
    \begin{align*}
        N(\tau) &= \frac{\operatorname{vol}(\mathbb{S}^{n-2})\operatorname{vol}(\partial \Omega)}{n-1} 
        \biggl[ 
            \biggl( \frac{\lambda + 3\mu}{2\mu (\lambda + \mu)} \biggr)^{n-1} + \frac{1}{(2\mu)^{n-1}} + \frac{n-2}{\mu^{n-1}} + \frac{1}{\alpha^{n-1}}
        \biggr] \tau^{n-1} \\
        &\quad + o(\tau^{n-1}) \quad \text{as}\ \tau \to +\infty.
    \end{align*}
\end{corollary}

\addvspace{2mm}

The main ideas of this paper are as follows. Since the thermoelastic Dirichlet-to-Neumann map is a pseudodifferential operator defined on the boundary of a manifold, there is an effective method (see \cite{Liu19}) to calculate its full symbol. It is a novel idea of this paper that we separate the full symbols of the corresponding differential (or pseudodifferential) operators into two parts by analyzing the characteristics of the thermoelastic operator, one is the pure elastic term which has been obtained in \cite{Liu19}, the other is related to the thermal effect. Thus we can greatly simplify the complicated computations of the full symbols and the spectral invariants. By combining the two parts we can obtain the spectral invariants of the thermoelastic Dirichlet-to-Neumann map. Firstly, we flat the boundary and induce a Riemannian metric in a neighborhood of the boundary and give an important local representation of the thermoelastic Dirichlet-to-Neumann map $\Lambda_{g} = A(-\frac{\partial }{\partial x_n})-D$ in boundary normal coordinates, where $A$ and $D$ are two matrices (see \eqref{2.9}). We then look for the factorization
\begin{align*}
    A^{-1} \mathcal{L}_g 
    = I_{n+1}\frac{\partial^2 }{\partial x_n^2} + B \frac{\partial }{\partial x_n} + C
    = \Bigl(I_{n+1}\frac{\partial }{\partial x_n} + B - Q\Bigr)\Bigl(I_{n+1}\frac{\partial }{\partial x_n} + Q\Bigr),
\end{align*}
where $I_{n+1}$ is the $(n+1)\times (n+1)$ identity matrix, $B$, $C$ are two differential operators and $Q$ is a pseudodifferential operator. As a result, we obtain
\begin{align*}
    Q^2 - BQ - \Bigl[I_{n+1} \frac{\partial }{\partial x_n},Q\Bigr] + C = 0,
\end{align*}
where $[\cdot,\cdot]$ is the commutator. Next, let $b$, $c$ and $q$ be the full symbols of the operators $B$, $C$ and $Q$, respectively. We solve the full symbol equation
\begin{align*}
    \sum_{J} \frac{(-i)^{|J|}}{J !} \partial_{\xi}^{J}q \, \partial_{x^\prime}^{J}q - \sum_{J} \frac{(-i)^{|J|}}{J !} \partial_{\xi}^{J}b \, \partial_{x^\prime}^{J}q - \frac{\partial q}{\partial x_n} + c = 0,
\end{align*}
where the sum is over all multi-indices $J$, $\xi=(\xi_1,\dots,\xi_{n-1})$ and $x^\prime=(x_1,\dots,x_{n-1})$. Note that the above equation is a quadratic matrix equation. Generally, it can not be exactly solved, in other words, there is not a general formula of the solution represented by the coefficients of the matrix equation. Since these are $(n+1)\times(n+1)$ matrices in the above full symbol equation, so we develop the method of the previous work \cite{Liu19} to deal with the complicated matrix equation and the computations of symbols. Fortunately, by combining the method in \cite{Liu19} we get the exact solution in this setting. It follows that the trace $\operatorname{Tr} (e^{-t \Lambda_{g}})$ of the associated heat kernel $\textbf{\textit{K}}(t,x^{\prime},y^{\prime})$ admits the asymptotic expansion
\begin{align*}
    \sum_{k=1}^{\infty} e^{-t \tau_{k}}
    & = \int_{\partial \Omega} \textbf{\textit{K}}(t,x^{\prime},x^{\prime}) \,dS \\
    & = \int_{\partial \Omega}
    \bigg[ 
        \frac{1}{(2\pi)^{n-1}} \int_{T^{*}(\partial \Omega)} e^{i \langle x^{\prime} - x^{\prime},\xi \rangle}
        \bigg( 
            \frac{i}{2\pi}\int_{\mathcal{C}} e^{-t\tau}\sigma\bigl((\Lambda_g-\tau I)^{-1}\bigr) \,d\tau
        \bigg) \,d\xi
    \bigg] \,dS \\
    & \sim \sum_{k=0}^{\infty} a_k t^{-n+k+1} + \sum_{l=0}^{\infty} b_{l} t^l \log t \quad \text{as}\ t \to 0^{+},
\end{align*}
where $\mathcal{C}$ is a contour around the positive real axis and $\sigma\bigl((\Lambda_g-\tau I)^{-1}\bigr)$ is the full symbol of the resolvent operator $(\Lambda_g-\tau I)^{-1}$. Finally, we establish an effective procedure to calculate all the coefficients $a_k\ (0\leqslant k \leqslant n-1)$ and explicitly give the first two coefficients.

This paper is organized as follows. In Section \ref{s2} we give the expression of the thermoelastic Dirichlet-to-Neumann map in boundary normal coordinates. In Section \ref{s3} we derive a factorization of the thermoelastic operator and get the full symbols of the pseudodifferential operators. Finally, Section \ref{s4} is devoted to prove Theorem \ref{thm1.1}.

\addvspace{10mm}

\section{Preliminaries}\label{s2}

\addvspace{5mm}

In the local coordinates $\{x_i\}_{i=1}^n$, we denote by $\bigl\{\frac{\partial}{\partial x_i}\bigr\}_{i=1}^n$ a natural basis for the tangent space $T_x \Omega$ at the point $x \in \Omega$. In what follows, we will use the Einstein summation convention, the Greek indices run from 1 to $n-1$, whereas Roman indices run from 1 to $n$, unless otherwise specified. Then the Riemannian metric is given by $g = g_{ij} \,dx_i\, dx_j$. Let $\nabla$ be the Levi-Civita connection on $\Omega$, $\nabla_i = \nabla_{\frac{\partial}{\partial x_i}}$ and $\nabla^i= g^{ij} \nabla_j$. The gradient operator is denoted by
\begin{equation}\label{2.1}
    \operatorname{grad} f = \nabla^i f \frac{\partial}{\partial x_i}
    = g^{ij} \frac{\partial f}{\partial x_i} \frac{\partial}{\partial x_j},\quad f \in C^{\infty}(\Omega),
\end{equation}
where $(g^{ij}) = (g_{ij})^{-1}$. The divergence operator is denoted by
\begin{equation}\label{2.2}
    \operatorname{div} \textbf{\textit{u}} = \nabla_i u^i = \frac{\partial u^i}{\partial x_i} + \Gamma^j_{ij} u^i,\quad \textbf{\textit{u}}=u^i \frac{\partial}{\partial x_i}\in \mathfrak{X}  (\Omega),
\end{equation}
where the Christoffel symbols $\Gamma^{k}_{ij} = \frac{1}{2} g^{kl} \bigl(\frac{\partial g_{jl}}{\partial x_i} + \frac{\partial g_{il}}{\partial x_j} - \frac{\partial g_{ij}}{\partial x_l}\bigr)$. It follows from \eqref{2.1} and \eqref{2.2} that
\begin{align}\label{2.3}
    \operatorname{grad}\operatorname{div} \textbf{\textit{u}} =  g^{ij}
    \Bigl( \frac{\partial^2 u^k}{\partial x_j \partial x_k} +  \Gamma^l_{kl}\frac{\partial u^k}{\partial x_j} +  \frac{\partial \Gamma^l_{kl}}{\partial x_j} u^k \Bigr) \frac{\partial}{\partial x_i}.
\end{align}

Recall that the components of the curvature tensor
\begin{align}\label{2.4}
    R^{l}_{ijk} 
    = \frac{\partial \Gamma^{l}_{jk}}{\partial x_i} - \frac{\partial \Gamma^{l}_{ik}}{\partial x_j} + \Gamma^{h}_{jk} \Gamma^{l}_{ih} - \Gamma^{h}_{ik} \Gamma^{l}_{jh},
\end{align}
and $R_{ijkl} = g_{lm} R^{m}_{ijk}$, the components of the the Ricci tensor
\begin{align}\label{2.5}
    R_{ij} = g^{kl} R_{iklj},
\end{align}
and $R=g^{ij}R_{ij}$ is the scalar curvature. The Bochner Laplacian is given by
\begin{align*}
    \Delta_{B}\textbf{\textit{u}} =(\nabla^j \nabla_j u^i)\frac{\partial}{\partial x_i} ,\quad \textbf{\textit{u}}=u^i \frac{\partial}{\partial x_i}\in \mathfrak{X}  (\Omega).
\end{align*}
Then combining \eqref{2.4} and \eqref{2.5}, one has
\begin{align}\label{2.6}
    (\Delta_{B}\textbf{\textit{u}})^i 
    & = \Delta_g u^i + g^{kl} 
    \Bigl(
        2\Gamma^i_{jk}\frac{\partial u^j}{\partial x_l} + \frac{\partial \Gamma^i_{jk}}{\partial x_l}u^j + (\Gamma^i_{hl}\Gamma^h_{jk} - \Gamma^h_{kl}\Gamma^i_{jh}) u^j
    \Bigr) \\
    & = \Delta_g u^i - \operatorname{Ric}(\textbf{\textit{u}})^i + g^{kl} \Bigl( 2\Gamma^i_{jk} \frac{\partial u^j}{\partial x_l} + \frac{\partial \Gamma^i_{kl}}{\partial x_j} u^j \Bigr), \notag
\end{align}
where $\operatorname{Ric}(\textbf{\textit{u}})^i = g^{ij}R_{jk} u^k$ and the Laplace--Beltrami operator is given by
\begin{equation}\label{2.7}
    \Delta_{g} f
    = g^{ij}\Bigl(\frac{\partial^2 f}{\partial x_i \partial x_j} - \Gamma^k_{ij}\frac{\partial f}{\partial x_k}\Bigr), \quad f \in C^{\infty}(\Omega).
\end{equation}

\addvspace{2mm}

Here we briefly introduce the construction of geodesic coordinates with respect to the boundary (see \cite{LeeUhlm89} or \cite[p.\,532]{Taylor11.2}). For each point $x^{\prime} \in \partial \Omega$, let $\gamma_{x^{\prime}}:[0,\varepsilon)\to \bar{\Omega}$ denote the unit-speed geodesic starting at $x^{\prime}$ and normal to $\partial \Omega$. If $ x^{\prime} := \{x_{1}, \ldots, x_{n-1}\}$ are any local coordinates for $\partial \Omega$ near $x_0 \in \partial \Omega$, we can extend them smoothly to functions on a neighborhood of $x_0$ in $\bar{\Omega}$ by letting them be constant along each normal geodesic $\gamma_{x^{\prime}}$. If we then define $x_n$ to be the parameter along each $\gamma_{x^{\prime}}$, it follows easily that $\{x_{1}, \ldots, x_{n}\}$ form coordinates for $\bar{\Omega}$ in some neighborhood of $x_0$, which we call the boundary normal coordinates determined by $\{x_{1}, \ldots, x_{n-1}\}$. In these coordinates $x_n>0$ in $\Omega$, and $\partial \Omega$ is locally characterized by $x_n=0$. A standard computation shows that the metric then has the form $g = g_{\alpha\beta} \,dx_{\alpha} \,dx_{\beta} + dx_{n}^{2}$ and
\begin{align*}
    & g_{jk}(x_0) = g^{jk}(x_0) = \delta_{jk},\quad 1 \leqslant j,k \leqslant n, \\
    & \frac{\partial g_{jk}}{\partial x_\alpha}(x_0) = \frac{\partial g^{jk}}{\partial x_\alpha}(x_0) = 0, \quad 1 \leqslant \alpha \leqslant n-1,
\end{align*}
where $\delta_{jk}$ is the Kronecker delta. Furthermore, we can choose a frame that diagonalizes the second fundamental form such that at the $x_0 \in \partial \Omega$,
\begin{align*}
    H =\sum_\alpha \kappa_\alpha 
    = -\Gamma^{\alpha}_{n\alpha}
    = -\frac{1}{2}\frac{\partial g_{\alpha\alpha}}{\partial x_n},
\end{align*}
where $\kappa_\alpha\ (1\leqslant \alpha \leqslant n-1)$ are the principal curvatures, $H$ is the mean curvature of $\partial \Omega$.

For the sake of simplicity, we denote by $I_n$ the $n\times n$ identity matrix,
\begin{align*}
    (a^\alpha_\beta):=
    \begin{pmatrix}
        a^1_1 & \dots & a^1_{n-1} \\
        \vdots & \ddots & \vdots \\
        a^{n-1}_1 & \dots & a^{n-1}_{n-1} \\
    \end{pmatrix},
\end{align*}
and
\begin{align*}
    \begin{pmatrix}
        (a^j_k) & (b^j) \\[1mm]
        (c_k) & d
    \end{pmatrix}
    :=
    \begin{pmatrix}
        (a^\alpha_\beta) & (a^\alpha_n) & (b^\alpha) \\[1mm]
        (a^n_\beta) & a^n_n & b^n \\[1mm]
        (c_\beta) & c_n & d
    \end{pmatrix}
    =
    \begin{pmatrix}
        a^1_1 & \dots & a^1_n & b^1 \\
        \vdots & \ddots & \vdots & \vdots \\
        a^n_1 & \dots & a^n_n & b^n \\
        c_1 & \dots & c_n & d \\
    \end{pmatrix},
\end{align*}
where $1 \leqslant j,k \leqslant n$ and $ 1 \leqslant \alpha,\beta \leqslant n-1$.

\begin{proposition}
    In the boundary normal coordinates, the thermoelastic Dirichlet-to-Neumann map $\Lambda_{g}$ can be written as
    \begin{align}\label{2.8}
        \Lambda_{g} = A\Bigl(-\frac{\partial }{\partial x_n}\Bigr)-D,
    \end{align}
    where
    \begin{align}\label{2.9}
        A=
        \begin{pmatrix}
            \mu I_{n-1} &0 &0  \\
            0& \lambda+2\mu &0\\
            0& 0& \alpha
        \end{pmatrix},\quad 
        D=
        \begin{pmatrix}
            0 & \mu \bigl(g^{\alpha\beta}\frac{\partial }{\partial x_\beta}\bigr) & 0\\
            \lambda \bigl(\frac{\partial }{\partial x_\beta} + \Gamma^\alpha_{\alpha\beta} \bigr) & \lambda \Gamma^\alpha_{\alpha n} & -\beta\\
            0 & 0 & 0
        \end{pmatrix}.
    \end{align}
\end{proposition}

\addvspace{2mm}

\begin{proof}
    By the definition of the stress tensor \eqref{1.7} (see \cite[p.\,562]{Taylor11.3}), we have
    \begin{align}\label{2.12}
        (S\textbf{\textit{u}})^i_j = \nabla^i u_j + \nabla_j u^i,
    \end{align}
    where $u_j=g_{ij}u^i$. Then
    \begin{align*}
        (\nu S\textbf{\textit{u}})^i =\nu^j (S\textbf{\textit{u}})^i_j
        &=\nu^j(\nabla^i u_j + \nabla_j u^i)
    \end{align*}
    In the boundary normal coordinates, we take $\nu=(0,\dots,0,-1)^t$ and $\partial_\nu = -\partial_{x_n}$. In particular, $u_n=u^n$ since $g_{in}=\delta_{in}$ in the boundary normal coordinates. We get
    \begin{align*}
        (\nu S\textbf{\textit{u}})^i = -(\nabla^i u_n + \nabla_n u^i).
    \end{align*}
    Note that $\Gamma^n_{nk}=\Gamma^k_{nn}=0$ and $g^{\alpha\beta}\Gamma^n_{\beta\gamma}+\Gamma^\alpha_{n\gamma}=0$ in the boundary normal coordinates. Thus
    \begin{align}\label{2.10}
        (\nu S\textbf{\textit{u}})^\alpha
        &=-(\nabla^\alpha u_n + \nabla_n u^\alpha)\\
        &=-\Bigl[g^{\alpha\beta}\Bigl(\frac{\partial u^n}{\partial x_\beta}+\Gamma^n_{\beta\gamma}u^\gamma\Bigr)+\frac{\partial u^\alpha}{\partial x_n}+\Gamma^\alpha_{n\gamma}u^\gamma\Bigr]\notag\\
        &=-g^{\alpha\beta}\frac{\partial u^n}{\partial x_\beta}-\frac{\partial u^\alpha}{\partial x_n},\notag\\
        (\nu S\textbf{\textit{u}})^n 
        &=-(\nabla^n u_n + \nabla_n u^n)=-2\frac{\partial u^n}{\partial x_n}.\label{2.11}
    \end{align}
    Hence, we immediately obtain \eqref{2.8} by combining \eqref{1.4}, \eqref{2.1}, \eqref{2.2}, \eqref{2.10} and \eqref{2.11}.
\end{proof}

\addvspace{10mm}

\section{Symbols of the Pseudodifferential Operators}\label{s3}

\addvspace{5mm}

In the boundary normal coordinates, we write the Laplace--Beltrami operator as
\begin{align}\label{3.1}
    \Delta_g
    & = \frac{\partial^2 }{\partial x_n^2} + \Gamma^\alpha_{n\alpha} \frac{\partial }{\partial x_n} + g^{\alpha\beta} \frac{\partial^2}{\partial x_\alpha\partial x_\beta} +
    \biggl(
        g^{\alpha\beta} \Gamma^\gamma_{\gamma\alpha} + \frac{\partial g^{\alpha\beta}}{\partial x_\alpha}
    \biggr)
    \frac{\partial }{\partial x_\beta}.
\end{align}
According to \eqref{1.2}, \eqref{2.1}, \eqref{2.2}, \eqref{2.3}, \eqref{2.6}, \eqref{2.7} and \eqref{3.1}, we deduce that (cf. \cite{Liu19})
\begin{align}\label{3.2}
    A^{-1} \mathcal{L}_g = I_{n+1}\frac{\partial^2 }{\partial x_n^2} + B \frac{\partial }{\partial x_n} + C,
\end{align}
where $A$ is given by \eqref{2.9}, $B=B_1+B_0$, $C =C_2+C_1+C_0$, and
\begin{align*}
    B_1&=(\lambda+\mu)
    \begin{pmatrix}
        0 & \frac{1}{\mu}\bigl(g^{\alpha\beta}\frac{\partial}{\partial x_{\beta}}\bigr) & 0 \\
        \frac{1}{\lambda+2\mu}\bigl(\frac{\partial}{\partial x_{\beta}}\bigr) & 0 & 0\\
        0 & 0 & 0
    \end{pmatrix},\quad 
    B_0=
    \begin{pmatrix}
        \Gamma^\alpha_{n\alpha}I_{n-1}+2(\Gamma^\alpha_{n\beta}) & 0 & 0 \\[1mm]
        \frac{\lambda+\mu}{\lambda+2\mu}\bigl(\Gamma^\alpha_{\alpha\beta}\bigr) & \Gamma^\alpha_{n\alpha} & - \frac{\beta}{\lambda+2\mu}\\[1mm]
        0 & \frac{i\omega\beta\theta_0}{\alpha} & \Gamma^\alpha_{n\alpha}
    \end{pmatrix},\\[1mm]
    C_2&=
    \begin{pmatrix}
        g^{\alpha\beta}\frac{\partial^2 }{\partial x_\alpha \partial x_\beta}I_{n-1} + \frac{\lambda+\mu}{\mu}\bigl(g^{\alpha\gamma}\frac{\partial^2 }{\partial x_\alpha \partial x_\beta}\bigr) & 0 & 0\\
        0 & \frac{\mu}{\lambda+2\mu}g^{\alpha\beta}\frac{\partial^2 }{\partial x_\alpha \partial x_\beta} & 0\\
        0 & 0 & g^{\alpha\beta}\frac{\partial^2 }{\partial x_\alpha \partial x_\beta}
    \end{pmatrix},\\[1mm]
    C_1&=
    \begin{pmatrix}
        \bigl(g^{\alpha\beta}\Gamma^\gamma_{\alpha\gamma}+\frac{\partial g^{\alpha\beta}}{\partial x_{\alpha}}\bigr)\frac{\partial }{\partial x_{\beta}} I_{n-1} & 0 & 0 \\
        0 & \frac{\mu}{\lambda+2\mu}\bigl(g^{\alpha\beta}\Gamma^\gamma_{\alpha\gamma}+\frac{\partial g^{\alpha\beta}}{\partial x_{\alpha}}\bigr)\frac{\partial }{\partial x_{\beta}} & 0 \\
        0 & 0 & \bigl(g^{\alpha\beta}\Gamma^\gamma_{\alpha\gamma}+\frac{\partial g^{\alpha\beta}}{\partial x_{\alpha}}\bigr)\frac{\partial }{\partial x_{\beta}}
    \end{pmatrix}\\[1mm]
    &\quad +\frac{\lambda+\mu}{\mu}
    \begin{pmatrix}
        \bigl(g^{\beta\gamma}\Gamma^\alpha_{\alpha\rho}\frac{\partial }{\partial x_{\beta}}\bigr) & \bigl(g^{\beta\gamma}\Gamma^\alpha_{n\alpha}\frac{\partial }{\partial x_{\beta}}\bigr) & 0\\
        0 & 0 & 0\\
        0 & 0 & 0
    \end{pmatrix}\\[1mm]
    &\quad +
    \begin{pmatrix}
        2\bigl(g^{\alpha\beta}\Gamma^\rho_{\beta\gamma}\frac{\partial }{\partial x_{\alpha}}\bigr) & 2\bigl(g^{\alpha\beta}\Gamma^\rho_{\beta n}\frac{\partial }{\partial x_{\alpha}}\bigr) & -\frac{\beta}{\mu}\bigl(g^{\alpha\beta}\frac{\partial }{\partial x_{\beta}}\bigr) \\[2mm]
        \frac{2\mu}{\lambda+2\mu}\bigl(g^{\alpha\beta}\Gamma^n_{\beta\gamma}\frac{\partial }{\partial x_{\alpha}}\bigr) & 0 & 0\\[2mm]
        \frac{i\omega\beta\theta_0}{\alpha}\bigl(\frac{\partial }{\partial x_{\beta}}\bigr) & 0 & 0
    \end{pmatrix},\\[1mm]
    C_0 &=(\lambda+\mu)
    \begin{pmatrix}
        \frac{1}{\mu}\bigl(g^{\gamma\rho}\frac{\partial \Gamma^\alpha_{\alpha\beta}}{\partial x_\rho}\bigr) & \frac{1}{\mu}\bigl(g^{\gamma\rho}\frac{\partial \Gamma^\alpha_{\alpha n}}{\partial x_\rho}\bigr) & 0\\[2mm]
        \frac{1}{\lambda+2\mu}\bigl(\frac{\partial \Gamma^\alpha_{\alpha\beta}}{\partial x_n}\bigr) & \frac{1}{\lambda+2\mu}\frac{\partial \Gamma^\alpha_{\alpha n}}{\partial x_n} & 0\\[2mm]
        0 & 0 & 0
    \end{pmatrix}
    +
    \begin{pmatrix}
        \bigl(g^{ml}\frac{\partial \Gamma^\alpha_{ml}}{\partial x_\beta}\bigr) & \bigl(g^{ml}\frac{\partial \Gamma^\alpha_{ml}}{\partial x_n}\bigr) & 0\\[2mm]
        \frac{\mu}{\lambda+2\mu}\bigl(g^{ml}\frac{\partial \Gamma^n_{ml}}{\partial x_\beta}\bigr) & \frac{\mu}{\lambda+2\mu}g^{ml}\frac{\partial \Gamma^n_{ml}}{\partial x_n} & 0\\[2mm]
        0 & 0 & 0
    \end{pmatrix}\\[1mm]
    &\quad +
    \begin{pmatrix}
        \frac{\rho\omega^2}{\mu}I_{n-1} & 0 & 0\\[2mm]
        0 & \frac{\rho\omega^2}{\lambda+2\mu} & 0\\[2mm]
        \frac{i\omega\beta\theta_0}{\alpha}(\Gamma^\alpha_{\alpha\beta}) & \frac{i\omega\beta\theta_0}{\alpha}\Gamma^\alpha_{\alpha n} & \frac{i\omega \alpha}{\alpha}
    \end{pmatrix}.
\end{align*}

\addvspace{2mm}

We then derive the microlocal factorization of the thermoelastic operator $\mathcal{L}_g$.
\begin{proposition}\label{prop3.1}
    There exists a pseudodifferential operator $Q(x,\partial_{x^\prime})$ of order one in $x^\prime$ depending smoothly on $x_n$ such that
    \begin{align*}
        A^{-1}\mathcal{L}_g
        = \Bigl(I_{n+1}\frac{\partial }{\partial x_n} + B - Q\Bigr)\Bigl(I_{n+1}\frac{\partial }{\partial x_n} + Q\Bigr)
    \end{align*}
    modulo a smoothing operator.
\end{proposition}

\addvspace{2mm}

\begin{proof}
    It follows from \eqref{3.2} that
    \begin{align*}
        I_{n+1}\frac{\partial^2 }{\partial x_n^2} + B \frac{\partial }{\partial x_n} + C 
        = \Bigl(I_{n+1}\frac{\partial }{\partial x_n} + B - Q\Bigr)\Bigl(I_{n+1}\frac{\partial }{\partial x_n} + Q\Bigr).
    \end{align*}
    Equivalently,
    \begin{align}\label{3.3}
        Q^2 - BQ - \Bigl[I_{n+1}\frac{\partial }{\partial x_n},Q\Bigr] + C = 0,
    \end{align}
    where the commutator $\big[I_{n+1}\frac{\partial }{\partial x_n},Q\big]$ is defined by
    \begin{align*}
        \Bigl[I_{n+1}\frac{\partial }{\partial x_n},Q\Bigr]f
        = I_{n+1}\frac{\partial }{\partial x_n}(Qf) - Q \Bigl(I_{n+1}\frac{\partial }{\partial x_n}\Bigr)f 
        = \frac{\partial Q}{\partial x_n}f,\quad f\in C^{\infty}(\Omega).
    \end{align*}

    Let $q = q(x,\xi)$ be the full symbol of the operator $Q$, we write $q(x,\xi) \sim \sum_{j\leqslant 1} q_j(x,\xi)$ with $q_j(x,\xi)$ homogeneous of degree $j$ in $\xi$. Let $b(x,\xi)=b_1(x,\xi) + b_0(x,\xi)$ and $c(x,\xi) = c_2(x,\xi) + c_1(x,\xi) + c_0(x,\xi)$ be the full symbols of $B$ and $C$, respectively. We denote $|\xi|^2=g^{\alpha\beta} \xi_\alpha\xi_\beta$, thus $b_0=B_0$, $c_0 =C_0$ and
    \begin{align*}
        &b_1=i(\lambda+\mu)
        \begin{pmatrix}
            0 & \frac{1}{\mu}(g^{\alpha\beta}\xi_\beta) & 0 \\
            \frac{1}{\lambda+2\mu}(\xi_\beta) & 0 & 0\\
            0 & 0 & 0
        \end{pmatrix},\
        c_2=
        \begin{pmatrix}
            -|\xi|^2 I_{n-1} - \frac{\lambda+\mu}{\mu}(g^{\alpha\gamma}\xi_\alpha \xi_\beta) & 0 & 0\\
            0 & -\frac{\mu}{\lambda+2\mu}|\xi|^2 & 0\\
            0 & 0 & -|\xi|^2
        \end{pmatrix},\\[1mm]
        &c_1=
        \begin{pmatrix}
            i\bigl(g^{\alpha\beta}\Gamma^\gamma_{\alpha\gamma}+\frac{\partial g^{\alpha\beta}}{\partial x_{\alpha}}\bigr)\xi_{\beta} I_{n-1} & 0 & 0 \\
            0 & \frac{i\mu}{\lambda+2\mu}\bigl(g^{\alpha\beta}\Gamma^\gamma_{\alpha\gamma}+\frac{\partial g^{\alpha\beta}}{\partial x_{\alpha}}\bigr)\xi_{\beta} & 0 \\
            0 & 0 & i\bigl(g^{\alpha\beta}\Gamma^\gamma_{\alpha\gamma}+\frac{\partial g^{\alpha\beta}}{\partial x_{\alpha}}\bigr)\xi_{\beta}
        \end{pmatrix}\\[1mm]
        &\quad +\frac{i(\lambda+\mu)}{\mu}
        \begin{pmatrix}
            (g^{\beta\gamma}\Gamma^\alpha_{\alpha\rho}\xi_{\beta}) & (g^{\beta\gamma}\Gamma^\alpha_{n\alpha}\xi_{\beta}) & 0\\
            0 & 0 & 0\\
            0 & 0 & 0
        \end{pmatrix} +
        \begin{pmatrix}
            2i(g^{\alpha\beta}\Gamma^\rho_{\beta\gamma}\xi_{\alpha}) & 2i(g^{\alpha\beta}\Gamma^\rho_{\beta n}\xi_{\alpha}) & -\frac{i\beta}{\mu}(g^{\alpha\beta}\xi_{\beta}) \\[2mm]
            \frac{2i\mu}{\lambda+2\mu}(g^{\alpha\beta}\Gamma^n_{\beta\gamma}\xi_{\alpha}) & 0 & 0\\[2mm]
            -\frac{\omega\beta\theta_0}{\alpha}(\xi_{\beta}) & 0 & 0
        \end{pmatrix}.
    \end{align*} 

    \addvspace{2mm}
    
    Recall that if $F$ and $G$ are two pseudodifferential operators with full symbols $f$ and $g$, respectively, then the full symbol $\sigma(FG)$ of the operator $FG$ is given by (see \cite[p.\,11]{Taylor11.2}, or \cite[p.\,71]{Hormander85.3}, see also \cite{Grubb86,Treves80})
    \begin{align*}
        \sigma(FG)\sim \sum_{J} \frac{(-i)^{|J|}}{J !} \partial_{\xi}^{J}f \, \partial_{x}^{J}g,
    \end{align*}
    where the sum is over all multi-indices $J$. We conclude that the full symbol equation of \eqref{3.3} is
    \begin{equation}\label{3.4}
        \sum_{J} \frac{(-i)^{|J|}}{J !} \partial_{\xi}^{J}q \, \partial_{x^\prime}^{J}q - \sum_{J} \frac{(-i)^{|J|}}{J !} \partial_{\xi}^{J}b \, \partial_{x^\prime}^{J}q - \frac{\partial q}{\partial x_n} + c = 0.
    \end{equation}

    We shall determine $q_j$ recursively so that \eqref{3.4} holds modulo $S^{-\infty}$. Grouping the homogeneous terms of degree two in \eqref{3.4}, one has
    \begin{align*}
        q_1^2-b_1q_1+c_2=0.
    \end{align*}
    We use the method in \cite{Liu19} to solve the above quadratic matrix equation, then (cf. (3.21) in \cite{Liu19})
    \begin{align}\label{3.5}
        q_1=
        \begin{pmatrix}
            |\xi|I_{n-1} + \frac{s_{1}}{|\xi|} (g^{\alpha\beta}\xi_\alpha\xi_\gamma) & i s_{1} (g^{\alpha\beta}\xi_\beta) & 0\\[1mm]
            i s_{1} (\xi_\beta) & (1-s_{1})|\xi| & 0\\[1mm]
            0 & 0 & |\xi|
        \end{pmatrix},
    \end{align}
    where we choose that $q_1$ is positive-definite and $s_{1}=\frac{\lambda+\mu}{\lambda+3\mu}$. Note that there is an additional $(n+1,n+1)$-entry $|\xi|$, and $s_1$ is different from (3.21) in \cite{Liu19}. Grouping the homogeneous terms of degree one in \eqref{3.4}, we get the following Sylvester equation:
    \begin{align*}
        (q_1-b_1)q_0+q_0q_1=E_1,
    \end{align*}
    where
    \begin{align*}
        E_1=i\sum_\alpha\frac{\partial (q_1-b_1)}{\partial \xi_\alpha}\frac{\partial q_1}{\partial x_\alpha}+b_0q_1+\frac{\partial q_1}{\partial x_n} - c_1.
    \end{align*}
    Note that $\frac{\partial q_1}{\partial x_\alpha}(x_0)=0$ in the boundary normal coordinates, we find that $\frac{\partial |\xi|}{\partial x_n}(x_0) = \frac{1}{|\xi|}\sum_\alpha \kappa_\alpha \xi_\alpha^2$. Therefore (cf. (5.24) in \cite{Liu19})
    \begin{align*}
        q_0(x_0)=
        \frac{1}{2}\biggl(\frac{1}{|\xi|^2}\sum_\alpha \kappa_\alpha \xi_\alpha^2 - H\biggr)I_{n+1}
        +
        \begin{pmatrix}
            (\tilde{q_0})^\alpha_\beta & (\tilde{q_0})^\alpha_n & \frac{i\beta}{(\lambda+3\mu)|\xi|}(\xi_\alpha) \\[2mm]
            (\tilde{q_0})^n_\beta & (\tilde{q_0})^n_n & -\frac{\beta}{\lambda+3\mu} \\[2mm]
            \frac{\mu\omega\beta\theta_0}{\alpha(\lambda+3\mu)|\xi|}(\xi_\beta) & \frac{i\mu\omega\beta\theta_0}{\alpha(\lambda+3\mu)} & 0
        \end{pmatrix},
    \end{align*}
    where
    \begin{align*}
        (\tilde{q_0})^\alpha_\beta &= -\kappa_\alpha(\delta_{\alpha\beta}) + 
        \biggl\{
            \frac{s_{1}\kappa_\alpha}{|\xi|^2} + \frac{H}{|\xi|^2} \biggl[\frac{s_{1}^2}{4}\biggl(\frac{3}{s_{2}}+1\biggr)+\frac{s_{1}}{4}\biggl(\frac{1}{s_{2}}-1\biggr)\biggr]\\
            &\quad + \frac{1}{|\xi|^4} \biggl[\frac{s_{1}^3}{4}\biggl(\frac{1}{s_{2}}+1\biggr)+\frac{s_{1}^2}{4}\biggl(7-\frac{1}{s_{2}}\biggr)-s_{1}\biggr]\sum_\gamma\kappa_\gamma \xi_\gamma^2
        \biggr\} (\xi_\alpha\xi_\beta),\\[1mm]
        (\tilde{q_0})^\alpha_n &=
        \biggl\{
            i(1+s_{1}) \frac{\kappa_\alpha}{|\xi|} + \frac{iH}{|\xi|}\biggl[\frac{s_{1}^2}{4}\biggl(\frac{3}{s_{2}}+1\biggr)-\frac{s_{1}}{4}\biggl(\frac{1}{s_{2}}+1\biggr)-\frac{1}{2}\biggl(1-\frac{1}{s_{2}}\biggr)\biggr] \\
            &\quad + \frac{i}{|\xi|^3}\biggl[\frac{s_{1}}{4}\biggl(\frac{1}{s_{2}}+7\biggr)-\frac{s_{1}^2}{2}\biggl(\frac{1}{s_{2}}-3\biggr)-\frac{s_{1}^3}{4}\biggl(3-\frac{1}{s_{2}}\biggr)\biggr]\sum_\gamma\kappa_\gamma \xi_\gamma^2
        \biggr\} (\xi_\alpha),\\[1mm]
        (\tilde{q_0})^n_\beta &=
        \biggl[
            -is_2(1+s_{1}) \frac{\kappa_\beta}{|\xi|} - \frac{iH}{4|\xi|}\big(s_{1}(1+s_{2})+s_{1}^2(3+s_{2})\big) \\
            &\quad - \frac{i}{4|\xi|^3} \big(2s^2(3s_2-1)+s_{1}^3(s_{2}+1)+s_{1}(1-3s_2)\big) \sum_\gamma\kappa_\gamma \xi_\gamma^2
        \biggr] (\xi_\beta),\\[1mm]
        (\tilde{q_0})^n_n &=
        \frac{H}{4} \big( s_{1}^2(3+s_{2})-s_{1}(1-s_{2}) \big)+\frac{1}{4|\xi|^2} \big(s_{1}^3(1+s_{2})-s_{1}^2(3-5s_2)-4s_1 s_2 \big) \sum_\gamma\kappa_\gamma \xi_\gamma^2,\\
        s_{2}&=\frac{\mu}{\lambda+2\mu}.
    \end{align*}
    Note that $s_2$ and the signs of the principal curvatures are different from (5.24) and (5.25) in \cite{Liu19}. Grouping the homogeneous terms of degree zero in \eqref{3.4}, we get
    \begin{align*}
        (q_1-b_1)q_{-1}+q_{-1}q_1=E_0,
    \end{align*}
    where
    \begin{align*}
        E_0=-q_0^2+i\sum_\alpha\Bigl(\frac{\partial (q_1-b_1)}{\partial \xi_\alpha}\frac{\partial q_0}{\partial x_\alpha}+\frac{\partial q_0}{\partial \xi_\alpha}\frac{\partial q_1}{\partial x_\alpha}\Bigr)+\frac{1}{2}\sum_{\alpha,\beta}\frac{\partial^2q_1}{\partial \xi_\alpha \partial\xi_\beta}\frac{\partial^2q_1}{\partial x_\alpha \partial x_\beta}+b_0q_0 +\frac{\partial q_0}{\partial x_n} - c_0.
    \end{align*}

    Proceeding recursively, grouping the homogeneous terms of degree $-m\ (m\geqslant 1)$ in \eqref{3.4}, we get
    \begin{align*}
        (q_1-b_1)q_{-m-1}+q_{-m-1}q_1=E_{-m},
    \end{align*}
    where
    \begin{align*}
        E_{-m}= b_0q_{-m}+\frac{\partial q_{-m}}{\partial x_n} - i\sum_\alpha\frac{\partial b_1}{\partial \xi_\alpha}\frac{\partial q_{-m}}{\partial x_\alpha} - \sum_{\substack{-m \leqslant j,k \leqslant 1 \\ |J| = j + k + m}} \frac{(-i)^{|J|}}{J !} \partial_{\xi}^{J} q_j\, \partial_{x^\prime}^{J} q_k, \quad m \geqslant 0.
    \end{align*}
\end{proof}

\addvspace{2mm}

We get the full symbol of the operator $Q$ by Proposition \ref{prop3.1}. This implies that $Q$ has been obtained on $\partial \Omega$ modulo a smoothing operator.
\begin{proposition}
    In the boundary normal coordinates, the thermoelastic Dirichlet-to-Neumann map $\Lambda_{g}$ can be represented as
    \begin{align}\label{3.6}
        \Lambda_{g} = AQ-D
    \end{align}
    modulo a smoothing operator, where $A$ and $D$ are given by \eqref{2.9}.
\end{proposition}

\addvspace{2mm}

\begin{proof}
    In the boundary normal coordinates $(x^\prime,x_n)$ with $x_n\in[0,T]$. Since the principal symbol of the operator $\mathcal{L}_{g}$ is negative-definite, the hyperplane ${x_n = 0}$ is non-characteristic, hence $\mathcal{L}_{g}$ is partially hypoelliptic with respect to this boundary (see \cite[p.\,107]{Hormander64}). Therefore, the solution to the equation $\mathcal{L}_g \textbf{\textit{U}} = 0$ is smooth in normal variable, that is, $\textbf{\textit{U}}\in (C^\infty([0,T];\mathfrak{D}^\prime (\mathbb{R}^{n})))^{n+1}$ locally. From Proposition \ref{prop3.1}, we see that \eqref{1.1} is locally equivalent to the following system of equations for $\textbf{\textit{U}},\textbf{\textit{W}}\in (C^\infty([0,T];\mathfrak{D}^\prime(\mathbb{R}^{n})))^{n+1}$:
    \begin{align*}
        \Bigl(I_{n+1}\frac{\partial }{\partial x_n} + Q\Bigr)\textbf{\textit{U}} & = \textbf{\textit{W}}, \quad \textbf{\textit{U}}\big|_{x_n=0}=\textbf{\textit{V}},\\
        \Bigl(I_{n+1}\frac{\partial }{\partial x_n} + B - Q\Bigr)\textbf{\textit{W}} & =\textbf{\textit{Y}} \in (C^\infty([0,T];\mathfrak{D}^\prime (\mathbb{R}^{n})))^{n+1}.
    \end{align*}
    Inspired by \cite{Liu19} (cf. \cite{LeeUhlm89}), if we substitute $t = T-x_n$ for the second equation above, then we get a backwards generalized heat equation
    \begin{align*}
        \frac{\partial \textbf{\textit{W}}}{\partial t} -(B-Q)\textbf{\textit{W}}=-\textbf{\textit{Y}}.
    \end{align*}
    Since $\textbf{\textit{U}}$ is smooth in the interior of $\Omega$ by interior regularity for elliptic operator $\mathcal{L}_{g}$, it follows that $\textbf{\textit{W}}$ is also smooth in the interior of $\Omega$, and so $\textbf{\textit{W}}|_{x_n=T}$ is smooth. In view of the principal symbol of $Q$ is positive-definite, we get that the solution operator for this heat equation is smooth for $t > 0$ (see \cite[p.\,134]{Treves80}). Therefore
    \begin{align*}
        \frac{\partial \textbf{\textit{U}}}{\partial x_n} + Q\textbf{\textit{U}}=\textbf{\textit{W}}
    \end{align*}
    locally. If we set $\mathcal{R} \textbf{\textit{V}} = \textbf{\textit{W}}|_{\partial \Omega}$, then $\mathcal{R}$ is a smoothing operator and
    \begin{align*}
        \frac{\partial \textbf{\textit{U}}}{\partial x_n}\bigg|_{\partial \Omega}=-Q\textbf{\textit{U}}|_{\partial \Omega}+\mathcal{R}\textbf{\textit{V}}.
    \end{align*}
    Combining this and \eqref{2.8}, we obtain \eqref{3.6}.
\end{proof}

We now apply the method of \cite{Liu19} (cf. \cite{Grubb86,Seeley67}) to get the full symbol of the resolvent operator $(\Lambda_g-\tau I)^{-1}$. Let $\Phi(\tau)$ be a two-sided parametrix for $\Lambda_g - \tau I$, i.e., $\Phi(\tau)$ is a pseudodifferential operator of order $-1$ with parameter $\tau$ for which
\begin{align*}
    & \Phi(\tau)(\Lambda_g - \tau I) = I \mod OPS^{-\infty}, \\
    & (\Lambda_g - \tau I)\Phi(\tau) = I \mod OPS^{-\infty}.
\end{align*}

Let $\sigma(\Lambda_g) \sim \sum_{j\leqslant 1} p_j(x,\xi)$ be the full symbol of $\Lambda_g$. It follows from \eqref{3.6} that
\begin{align*}
    p_1&=Aq_1-d_1,\\
    p_0&=Aq_0-d_0,\\
    p_{-m}&=Aq_{-m},\quad m\geqslant 1,
\end{align*}
where
\begin{align*}
    d_1=
    \begin{pmatrix}
        0 & i\mu(g^{\alpha\beta}\xi_\beta) & 0\\
        i\lambda (\xi_\beta) & 0 & 0\\
        0 & 0 & 0
    \end{pmatrix},\quad
    d_0=
    \begin{pmatrix}
        0 & 0 & 0\\
        \lambda (\Gamma^\alpha_{\alpha\beta}) & \lambda \Gamma^\alpha_{\alpha n} & -\beta\\
        0 & 0 & 0
    \end{pmatrix}.
\end{align*}
Let $\sigma(\Phi(\tau)) \sim \sum_{j \leqslant -1} \phi_j(x,\xi,\tau)$ be the full symbol of $\Phi(\tau)$. Therefore
\begin{align}
    \phi_{-1} &= (p_1-\tau I)^{-1}, \label{3.7}\\
    \phi_{-2} & = -\phi_{-1}
    \biggl(
        p_0\phi_{-1} - i\sum_\alpha \frac{\partial p_1}{\partial \xi_\alpha} \frac{\partial \phi_{-1}}{\partial x_\alpha}
    \biggr),\label{3.8}\\
    \phi_{-1-m} &= -\phi_{-1}\! \sum_{\substack{-m \leqslant j \leqslant 1 \\ -m \leqslant k \leqslant -1 \\ |J| = j + k + m}} \frac{(-i)^{|J|}}{J !} \partial_{\xi}^{J} p_j \, \partial_{x^\prime}^{J} \phi_k, \quad m \geqslant 2.\label{3.9}
\end{align}

\addvspace{10mm}

\section{Computations of the Coefficients of the Asymptotic Expansion}\label{s4}

\addvspace{5mm}

Inspired by \cite{Liu19} (cf. \cite{Grubb86,Seeley69}), we will establish an effective procedure to calculate the first $n-1$ coefficients of the asymptotic expansion of heat trace for the thermoelastic Dirichlet-to-Neumann map.

\addvspace{2mm}

\begin{proof}[Proof of Theorem {\rm \ref{thm1.1}}]
According to the theory of elliptic equations (see \cite{Morrey66,Morrey58.1,Morrey58.2,Stewart74}), we see that $\Lambda_g$ can generate a strongly continuous generalized heat semigroup $e^{-t\Lambda_g}$ in a suitable space defined on $\partial \Omega$. Furthermore, there exists a parabolic kernel $\textbf{\textit{K}}(t,x^{\prime},y^{\prime})$ such that (see \cite{Browder60} or \cite[p.\,4]{Fried64})
\begin{align*}
    e^{-t\Lambda_g} \textbf{\textit{V}}(x^{\prime})
    = \int_{\partial \Omega} \textbf{\textit{K}}(t,x^{\prime},y^{\prime})\textbf{\textit{V}}(x^{\prime}) \,dS, \quad \textbf{\textit{V}} \in (L^{2}(\partial \Omega))^{n+1}.
\end{align*}
Let $\{\textbf{\textit{U}}_k\}_{k \geqslant 1}$ be the orthonormal eigenvectors corresponding to the eigenvalues $\{\tau_k\}_{k \geqslant 1}$. Then
\begin{align*}
    \textbf{\textit{K}}(t,x^{\prime},y^{\prime}) 
    = e^{-t\Lambda_g} \delta(x^{\prime} - y^{\prime}) 
    = \sum_{k=1}^{\infty} e^{-t\tau_k} \textbf{\textit{U}}_k(x^{\prime}) \otimes \textbf{\textit{U}}_k(y^{\prime}).
\end{align*}
This implies that
\begin{align*}
    \int_{\partial \Omega} \operatorname{Tr} \textbf{\textit{K}}(t,x^{\prime},x^{\prime}) \,dS
    = \sum_{k=1}^{\infty} e^{-t\tau_k}
\end{align*}
is a spectral invariant. On the other hand, the semigroup $e^{-t\Lambda_g}$ can also be represented as
\begin{align*}
    e^{-t\Lambda_g}
    = \frac{i}{2\pi}\int_{\mathcal{C}} e^{-t\tau}(\Lambda_g-\tau I)^{-1} \,d\tau,
\end{align*}
where $\mathcal{C}$ is a suitable curve in the complex plane in the positive direction around the spectrum of $\Lambda_g$, that is, $\mathcal{C}$ is a contour around the positive real axis. It follows that
\begin{align*}
    \textbf{\textit{K}}(t,x^{\prime},y^{\prime})
    & = e^{-t\Lambda_g} \delta(x^{\prime} - y^{\prime}) \\
    & = \frac{1}{(2\pi)^{n-1}} \int_{T^{*}(\partial \Omega)} e^{i \langle x^{\prime} - y^{\prime},\xi \rangle} 
    \bigg( 
        \frac{i}{2\pi}\int_{\mathcal{C}} e^{-t\tau}\sigma\big((\Lambda_g-\tau I)^{-1}\big) \,d\tau
    \bigg) \,d\xi \\
    & = \frac{1}{(2\pi)^{n-1}} \int_{\mathbb{R}^{n-1}} e^{i \langle x^{\prime} - y^{\prime},\xi \rangle}
    \bigg(
        \frac{i}{2\pi}\int_{\mathcal{C}} e^{-t\tau} \sum_{j \leqslant -1} \phi_{j} \,d\tau 
    \bigg) \,d\xi,
\end{align*}
the trace of the kernel $\textbf{\textit{K}}(t,x^{\prime},y^{\prime})$ is
\begin{align*}
    \operatorname{Tr}\textbf{\textit{K}}(t,x^{\prime},x^{\prime})
    = \frac{1}{(2\pi)^{n-1}} \int_{\mathbb{R}^{n-1}}
    \bigg(
        \frac{i}{2\pi} \int_{\mathcal{C}} e^{-t\tau} \sum_{j \leqslant -1} \operatorname{Tr} \phi_{j}\,d\tau 
    \bigg) \,d\xi.
\end{align*}
Therefore,
\begin{align*}
    \sum_{k=1}^{\infty} e^{-t\tau_k}
    = \int_{\partial \Omega} 
    \bigg[ 
        \frac{1}{(2\pi)^{n-1}} \int_{\mathbb{R}^{n-1}}
        \bigg(
            \frac{i}{2\pi}\int_{\mathcal{C}} e^{-t\tau} \sum_{j \leqslant -1} \operatorname{Tr} \phi_{j} \,d\tau 
        \bigg) \,d\xi
    \bigg] \,dS.
\end{align*}
We will calculate the asymptotic expansion of the trace of the semigroup $e^{-t\Lambda_g}$ as $t \to 0^{+}$. More precisely, we will figure out that
\begin{equation}\label{4.1}
    a_{k}(x) 
    = \frac{i}{(2 \pi)^{n}}\int_{\mathbb{R}^{n-1}}\int_{\mathcal{C}} e^{-t\tau} \operatorname{Tr} \phi_{-1-k} \,d\tau \,d\xi, \quad  0 \leqslant k \leqslant n-1.
\end{equation}

\addvspace{2mm}

(i) According to \eqref{3.7}, at the origin $x_0 \in \partial \Omega$ in the boundary normal coordinates, we have (cf. (5.19) in \cite{Liu19})
\begin{align*}
    \phi_{-1}=
    \begin{pmatrix}
        \displaystyle -\frac{1}{\tau-\mu|\xi|}\big[I_{n-1} - t_1 (\xi_\alpha\xi_\beta)\big] & \displaystyle \frac{i\mu(1-s_1)(1-t_1|\xi|^2)}{(\tau-\mu|\xi|)(\tau-s_4|\xi|)}(\xi_\alpha) & 0\\[4mm]
        \displaystyle \frac{i(\lambda-s_3)(1-t_1|\xi|^2)}{(\tau-\mu|\xi|)(\tau-s_4|\xi|)}(\xi_\beta) & \displaystyle \frac{\mu(1-s_1)(\lambda-s_3)(1-t_1|\xi|^2)|\xi|^2}{(\tau-\mu|\xi|)(\tau-s_4|\xi|)^2}-\frac{1}{\tau-s_4|\xi|} &0\\[4mm]
        0&0&\displaystyle  -\frac{1}{\tau-\alpha|\xi|}
    \end{pmatrix},
\end{align*}
where
\begin{align*}
    s_3=\frac{(\lambda+\mu)(\lambda+2\mu)}{\lambda+3\mu},\quad
    s_4=\frac{2\mu(\lambda+2\mu)}{\lambda+3\mu},\quad
    t_1= \frac{\mu}{2|\xi|}\biggl(\frac{1-2s_1}{\tau-2\mu s_1|\xi|}-\frac{1}{\tau-2\mu|\xi|}\biggr).
\end{align*}
Therefore,
\begin{align*}
    \operatorname{Tr} \phi_{-1}&=
    -\frac{1}{\tau-\alpha|\xi|}-\frac{n-1}{\tau-\mu|\xi|} -\frac{1}{\tau-s_4|\xi|} -\frac{\mu s_1|\xi|}{(\tau-2\mu|\xi|)(\tau-2\mu s_1|\xi|)} \\
    &\quad-\frac{\mu^2(1-2s_1)|\xi|^2}{(\tau-\mu|\xi|)(\tau-2\mu |\xi|)(\tau-2\mu s_1|\xi|)} -\frac{\mu^2(1-s_1)^2|\xi|^2}{(\tau-\mu|\xi|)(\tau-s_4|\xi|)^2} \\
    &\quad -\frac{\mu^3s_1(1-s_1)^2|\xi|^3}{(\tau-2\mu|\xi|)(\tau-2\mu s_1|\xi|)(\tau-s_4|\xi|)^2} -\frac{\mu^4(1-s_1)^2(1-2s_1)|\xi|^4}{(\tau-\mu|\xi|)(\tau-2\mu|\xi|)(\tau-s_4|\xi|)^2}.
\end{align*}
By applying residue theorem, we get
\begin{align*}
    \frac{i}{2\pi}\int_{\mathcal{C}}e^{-t\tau}\operatorname{Tr}\phi_{-1}\,d\tau =
    e^{-t\alpha |\xi|}+(n-2)e^{-t\mu|\xi|}+e^{-2t\mu|\xi|}+e^{-2t\mu s_1|\xi|}.
\end{align*}
Using the following integral formula
\begin{align}\label{4.2}
    \int_{\mathbb{R}^{n-1}}e^{-C|\xi|}\,d\xi = \frac{\Gamma(n-1)\operatorname{vol}(\mathbb{S}^{n-2})}{C^{n-1}},\quad C>0,\ n\geqslant 2,
\end{align}
we obtain
\begin{align*}
    \frac{i}{(2 \pi)^{n}}\int_{\mathbb{R}^{n-1}}\int_{\mathcal{C}} e^{-t\tau} \operatorname{Tr} \phi_{-1} \,d\tau \,d\xi
    =a_0(x)t^{1-n},
\end{align*}
where $a_0(x)$ is given in Theorem \ref{thm1.1}.

\addvspace{2mm}

(ii) Note that $\frac{\partial |\xi|}{\partial x_\alpha}(x_0)=0$ in the boundary normal coordinates, we observe that $\frac{\partial p_1}{\partial x_\alpha}(x_0)=0$ and $\frac{\partial \phi_{-1}}{\partial x_\alpha}(x_0)=0$. Then, by \eqref{3.8} we find that $\phi_{-2}=-\phi_{-1}p_0\phi_{-1}$ at $x_0\in \partial \Omega$. Hence (cf. (5.28) in \cite{Liu19})
\begin{align*}
    &\operatorname{Tr} \phi_{-2}
    =-\operatorname{Tr}\,(\phi_{-1}^2p_0)
    =\frac{\alpha}{2(\tau-\alpha |\xi|)^2}\biggl(H-\frac{1}{|\xi|^2}\sum_\alpha \kappa_\alpha\xi_\alpha^2\biggr) -\frac{t_2}{(\tau-s_4|\xi|)^2} \\
    &\quad +\frac{1}{(\tau-\mu|\xi|)^2}\biggl[\frac{1}{2}(n+1)\mu H-\frac{(n-1)\mu H}{2|\xi|^2}\sum_\alpha \kappa_\alpha\xi_\alpha^2-t_3-t_5-t_1t_4(t_1|\xi|^2-2)\biggr]\\
    &\quad +\frac{\mu(1-s_1)(\lambda-s_3)(1-t_1|\xi|^2)^2(t_2|\xi|^2+t_4)}{(\tau-\mu|\xi|)^2(\tau-s_4|\xi|)^2} 
    +\frac{2\mu t_2(1-s_1)(\lambda-s_3)(1-t_1|\xi|^2)|\xi|}{(\tau-\mu|\xi|)(\tau-s_4|\xi|)^3}\\
    &\quad   
    +(t_6+t_7)\biggl[\frac{(1-t_1|\xi|^2)^2}{(\tau-\mu|\xi|)^2(\tau-s_4|\xi|)} +\frac{1-t_1|\xi|^2}{(\tau-\mu|\xi|)(\tau-s_4|\xi|)^2}-\frac{\mu (1-s_1)(\lambda-s_3)(1-t_1|\xi|^2)^2|\xi|^2}{(\tau-\mu|\xi|)^2(\tau-s_4|\xi|)^3}\biggr]\\
    &\quad -\frac{\mu^2t_2(1-s_1)^2(\lambda-s_3)^2(1-t_1|\xi|^2)^2|\xi|^4}{(\tau-\mu|\xi|)^2(\tau-s_4|\xi|)^4},
\end{align*}
where
\begin{align*}
    t_2&=\frac{1}{2(\lambda+3\mu)^2}\biggl[2H(\lambda^3+4\lambda^2\mu+\lambda\mu^2-8\mu^3)+\mu(3\lambda^2+12\lambda\mu+13\mu^2)\frac{1}{|\xi|^2}\sum_\alpha \kappa_\alpha\xi_\alpha^2\biggr],\\
    t_3&=\frac{(\lambda+\mu)^2}{2(\lambda+3\mu)^2}\biggl[H(2\lambda+5\mu)-\frac{\mu}{|\xi|^2}\sum_\alpha \kappa_\alpha\xi_\alpha^2\biggr],\\
    t_4&=\frac{(\lambda+2\mu)^2}{(\lambda+3\mu)^2}\biggl[H(\lambda^2+2\lambda\mu-\mu^2)|\xi|^2+2\lambda\mu\sum_\alpha \kappa_\alpha\xi_\alpha^2\biggr],\\
    t_5&=\frac{\mu(\lambda+\mu)(3\lambda+5\mu)}{(\lambda+3\mu)^2|\xi|^2}\sum_\alpha \kappa_\alpha\xi_\alpha^2,\\
    t_6&=\frac{2\mu^2}{(\lambda+3\mu)^3}\biggl[H(\lambda+\mu)(\lambda+2\mu)^2|\xi|+\mu(3\lambda^2+11\lambda\mu+12\mu^2)\frac{1}{|\xi|}\sum_\alpha \kappa_\alpha\xi_\alpha^2\biggr],\\
    t_7&=\frac{2\mu^2}{(\lambda+3\mu)^4}\biggl[H(\lambda+\mu)(\lambda+3\mu)(\lambda+2\mu)^2|\xi| +\mu(\lambda+5\mu)(2\lambda^2+7\lambda\mu+7\mu^2)\frac{1}{|\xi|}\sum_\alpha \kappa_\alpha\xi_\alpha^2\biggr].
\end{align*}
Using \eqref{4.2} and the following formula
\begin{align*}
    \int_{\mathbb{R}^{n-1}}e^{-C|\xi|} \frac{\xi_{\alpha}^2}{|\xi|^2} \,d\xi = \frac{\Gamma(n-1)\operatorname{vol}(\mathbb{S}^{n-2})}{(n-1)C^{n-1}},\quad C>0,\ n\geqslant 2,
\end{align*}
and applying residue theorem again, we obtain 
\begin{align*}
    \frac{i}{(2 \pi)^{n}}\int_{\mathbb{R}^{n-1}}\int_{\mathcal{C}} e^{-t\tau} \operatorname{Tr} \phi_{-2} \,d\tau \,d\xi
    =a_1(x)t^{2-n},
\end{align*}
where $a_1(x)$ is given in Theorem \ref{thm1.1}.

\addvspace{1mm}

Generally, we can calculate the symbols $q_{-k}$ and $p_{-k}$ for $k\geqslant 1$, then we get $\phi_{-1-k}$ by applying \eqref{3.9}. Therefore, \eqref{4.1} gives all the coefficients $a_k(x)$ for $0 \leqslant k \leqslant n-1$.
\end{proof}

\addvspace{5mm}

\begin{remark}
    \begin{enumerate}[{\rm (i)}]
        \item It can be verified that $a_2(x)$ and $a_3(x)$ have the following forms, respectively,
        \begin{align*}
            a_2(x)=\frac{\Gamma(n-2)\operatorname{vol}(\mathbb{S}^{n-2})}{(2\pi)^{n-1}}
            \bigg(
                a_{2,1}H^2+a_{2,2}\sum_{\alpha} \kappa_\alpha^2+ a_{2,3}R+a_{2,4}R_{nn}
            \bigg)
        \end{align*}
        and
        \begin{align*}
            a_3(x)=&\frac{\Gamma(n-3)\operatorname{vol}(\mathbb{S}^{n-2})}{(2\pi)^{n-1}}
            \bigg(
                a_{3,1}H^3+a_{3,2}H\sum_{\alpha} \kappa_\alpha^2+ a_{3,3}\sum_{\alpha} \kappa_\alpha^3+a_{3,4}HR+a_{3,5}HR_{nn} \\
                &+ a_{3,6}\sum_\alpha \kappa_\alpha R_{\alpha\alpha}+a_{3,7}\sum_\alpha \kappa_\alpha R_{n\alpha\alpha n}+a_{3,8}\nabla_{n}R_{nn}
            \bigg),
        \end{align*}
        where $R,\, R_{ii}\,(1\leqslant i \leqslant n)$ and $R_{n\alpha\alpha n}\,(1\leqslant \alpha \leqslant n-1)$ denote the scalar curvature, the components of the Ricci tensor and the Riemann curvature tensor of $\Omega$, respectively, and the coefficients $a_{2,j},\,a_{3,j}$ are constants depend only on $\lambda$, $\mu$, $\alpha$ and $n$.
    
        \addvspace{2mm}
    
        \item Generally, for each $k\ (0 \leqslant k \leqslant n-1)$, $a_k(x)$ is a universal polynomial in terms of the metric, its inverse and their derivatives, the total number of derivatives in each term is $k$. It follows from {\rm \cite[Theorem 1.1.3]{Gilkey04}} that for $2 \leqslant k \leqslant n-1$, $a_k(x)$ can be written as a universal polynomial in terms of $\nabla_{\partial \Omega}^j h$ for $0 \leqslant j \leqslant k-1$ and $\nabla^j {\rm Rm}$ for $0 \leqslant j \leqslant k-2$, where $\nabla_{\partial \Omega}$ is the Levi-Civita connection on $\partial \Omega$, $h$ is the second fundamental form of $\partial \Omega$ and ${\rm Rm}$ is the Riemann curvature tensor of $\Omega$.
    \end{enumerate}
\end{remark}

\addvspace{10mm}

\section*{Acknowledgements}

\addvspace{5mm}

This research was supported by NNSF of China (11671033/A010802) and NNSF of China \\
(11171023/A010801).

\addvspace{10mm}

\addvspace{5mm}

\end{document}